\documentclass[10pt]{amsart}
\usepackage{euscript}
\usepackage{amssymb}
\usepackage{amsmath}
\usepackage{epic}
\usepackage{graphics}
\usepackage{epsfig}
\usepackage{color}

\usepackage{amscd,euscript}
\usepackage[frame,cmtip,curve,arrow,matrix,line,graph]{xy}

\numberwithin{equation}{section}
\newtheoremstyle{my}{1.5em}{0.5em}{\em}{}{\sc}{.}{0.5em}{}

\newtheorem{thm}{Theorem}[section]
\newtheorem{Theorem}[thm]{Theorem}
\newtheorem*{Theorem*}{Theorem}
\newtheorem{Corollary}[thm]{Corollary}

\newtheorem*{corollary*}{Corollary}
\newtheorem{Lemma}[thm]{Lemma}

\newtheorem{Proposition}[thm]{Proposition}

\newtheorem*{conjecture*}{Conjecture}

\newtheorem*{question*}{Question}

\newtheorem*{definitions*}{Definitions}

\newtheorem*{rem*}{Remark}
\newtheorem{Remark}[thm]{Remark}

\newtheorem*{remark*}{Remark}

\newtheorem*{remarks*}{Remarks}
\newtheorem*{example*}{Example}

\newtheorem*{examples*}{Examples}

\newtheorem*{convention*}{Convention}
\newtheorem*{conventions*}{Conventions}
\newtheorem{Hypothesis}[thm]{Hypothesis}

\newtheorem*{exercise*}{Exercise}
\newtheorem*{bibliographical-note*}{Bibliographical note}

\newcommand{\Acknowledgements}{{\em Acknowledgements.} }

\newcommand{\scrQ}{\EuScript{Q}}
\newcommand{\scrA}{\EuScript{A}}
\newcommand{\scrT}{\EuScript{T}}

\newcommand{\scrX}{\EuScript{X}}

\newcommand{\scrC}{\EuScript{C}}

\newcommand{\scrP}{\EuScript{P}}

\newcommand{\scrW}{\EuScript{W}}

\newcommand{\bF}{\mathbb{F}}
\newcommand{\bR}{\mathbb{R}}
\newcommand{\bZ}{\mathbb{Z}}
\newcommand{\bQ}{\mathbb{Q}}
\newcommand{\bC}{\mathbb{C}}
\newcommand{\bN}{\mathbb{N}}
\newcommand{\bP}{\mathbb{P}}

\newcommand{\Sym}{\mathrm{Sym}}
\newcommand{\Pic}{\mathrm{Pic}}
\newcommand{\id}{\mathrm{id}}

\renewcommand{\ker}{\mathrm{ker}}

\newcommand{\Hom}{\mathrm{Hom}}

\newcommand{\Conf}{\mathrm{Conf}}
\newcommand{\Jac}{\mathrm{Jac}}

\renewcommand{\hbar}{\overline{\frak{h}}}

\newcommand{\Aut}{\mathrm{Aut}}
\newcommand{\Ham}{\mathrm{Ham}}
\newcommand{\Symp}{\mathrm{Symp}}
\newcommand{\Diff}{\mathrm{Diff}}

\newcommand{\scrM}{\EuScript{M}}

\newcommand{\scrF}{\EuScript{F}}

\newcommand{\scrU}{\EuScript{U}}

\newcommand{\diag}{\mathrm{diag}}

\numberwithin{equation}{section}
\renewcommand{\geq}{\geqslant}

\newcommand{\tensor}{\otimes}

\newcommand{\lra}{\longrightarrow}

\newcommand{\C}{\mathbb C}

\renewcommand{\P}{\mathcal{P}}

\newcommand{\Z}{\mathbb{Z}}

\newcommand{\Tw}{\operatorname{Tw}}

\newcommand{\st}{\operatorname{st}}

\renewcommand{\sc}{\operatorname{sc}}

\newcommand{\Auteq}{\mathrm{Auteq}}

\newcommand{\Stab}{\operatorname{Stab}}

\title{Stability conditions in symplectic topology}
\author{Ivan Smith}
\address{Ivan Smith, Centre for Mathematical Sciences, University of Cambridge, Wilberforce Road, CB3 0WB, U.K.}
\email{is200@cam.ac.uk}
\date{November 2017. I.S. is partially funded by a Fellowship from EPSRC. Thanks to friends and colleagues at Cambridge for a decade and a half of moral support, intellectual stimulation and company lifting spirits.}

\begin{document}
\maketitle \thispagestyle{empty}

\begin{abstract} We discuss potential (largely speculative) applications of Bridgeland's theory of stability conditions to symplectic mapping class groups. \end{abstract}

\section{Introduction}

A symplectic manifold $(X,\omega)$ which is closed or convex at infinity  has a Fukaya category $\scrF(X,\omega)$, which packages the algebraic information held by moduli spaces of holomorphic discs in $X$ with Lagrangian boundary conditions \cite{Seidel:FCPLT, FO3}.  The associated category of perfect modules $D^{\pi}\scrF(X,\omega) = \scrF(X,\omega)^{per\!f}$ is a triangulated category, linear over the Novikov field $\Lambda$. The category is $\bZ$-graded whenever $2c_1(X)=0$.

Any $\bZ$-graded triangulated category $\scrC$  has an associated complex manifold $\Stab(\scrC)$ of stability conditions \cite{Bridgeland}, which carries an action of the group $\Auteq(\scrC)$ of triangulated autoequivalences of $\scrC$.  Computation of the space of stability conditions remains challenging, but a number of instructive examples are now available, and it is reasonable to wonder what the theory of stability conditions might say about symplectic topology. One direction in which one can hope for non-trivial applications concerns global structural features of the symplectic mapping class group.   The existing theory will need substantial  development for this to bear  fruit, so a survey risks being quixotic, but it provides a vantage-point from which to interpret some recent activity \cite{BS, Smith:quiver, SS:k3}.

\vspace{1em}

\noindent \Acknowledgements These notes draw on discussions or joint work with all of Denis Auroux, Arend Bayer, Tom Bridgeland, Jonny Evans, Daniel Huybrechts, Oscar Randal-Williams, Tony Scholl, Paul Seidel, Nick Sheridan and Michael Wemyss.   More generally, I am extremely indebted to all my collaborators, who have taught me more than I (and probably they) would have thought possible.

\section{Mapping class groups}

\subsection{Surfaces\label{Subsec:surfaces}} Consider a closed oriented surface $\Sigma_g$ of genus $g$. The mapping class group $\Gamma_g = \pi_0\Diff(\Sigma_g)$ has been the focus of an enormous amount of attention, from many viewpoints.  Here are a few of its salient features, see \cite{Farb-Margalit, Calegari:scl, Morita} for detailed references.  

\begin{enumerate}

\item $\Gamma_g$ is finitely presented.  The action on homology defines a natural surjection $\Gamma_g\to Sp_{2g}(\bZ)$ with kernel the Torelli group $I_g$; the latter is torsion-free, infinitely generated when $g=2$, finitely generated (but not known to be finitely presented) when $g>2$.

\item $\Gamma_g$ has finite rational cohomological dimension. It has finitely many conjugacy classes of finite subgroups; is generated by finitely many torsion elements; satisfies the Tits alternative; and is residually finite, in particular contains no non-trivial divisible element. 

\item There is a dynamical classification of mapping classes: given simple closed curves $\alpha, \beta$ on $\Sigma$, the geometric intersection numbers $\iota(\alpha, \phi^n(\beta))$  are either periodic, grow linearly or grow exponentially.   A random walk on the mapping class group will almost surely end on the last case. $\Gamma_g$ has non-trivial quasi-morphisms (and $scl$), hence does not obey a law.

\end{enumerate}
Here $scl$ denotes stable commutator length. Recall that a group $G$ \emph{obeys a law} if there is a word $w$ in a free group $\bF$ for which every homomorphism $\bF \to G$ sends $w\mapsto \id$. (Abelian, nilpotent and solvable groups obey laws; groups with non-vanishing $scl$ do not.)  It is worth pointing out that plenty of elementary things remain unknown: must a subgroup of $I_g$ generated by two elements be commutative or free?  

Broadly speaking, the above results might be divided into three categories: results (finite presentability, non-existence of non-trivial divisible elements, etc) which underscore the similarities between the mapping class group and arithmetic groups; results of a general group-theoretic or finiteness nature, in the vein of geometric group theory;  and results ($scl$, behaviour of random walks) of a more dynamical flavour, in some cases connecting via geometric intersection numbers rather directly to Floer theory.  These three directions suggest broad classes of question one might ask about the (symplectic) mapping class groups of higher-dimensional manifolds.

Much of what we know about $\Gamma_g$ comes from the fact that, although not a hyperbolic group when $g>1$ (it contains free abelian subgroups of rank $2g-1>1$),  it acts on geometrically meaningful spaces which are non-positively-curved in a suitable sense:  Teichm\"uller space, on the one hand, and the complex of curves on another (the latter is $\delta$-hyperbolic \cite{Masur-Minsky}).   Indeed, $\Gamma_g$ can be characterised via these actions, as the isometry group of the former for the Teichm\"uller metric (when $g>2$; for $g=2$ one should quotient by the hyperelliptic involution) \cite{Royden}, or as the simplicial isometry group of the latter \cite{Ivanov}.

\subsection{Higher-dimensional smooth manifolds}

There is no simply-connected manifold of dimension $>3$ for which the homotopy type of the diffeomorphism group $\Diff(M)$ is known completely. Nonetheless, there are many broad structural results concerning diffeomorphism groups and mapping class groups in high dimensions. 
Deep results in surgery theory \cite{Sullivan} imply that if $M$ is a simply-connected manifold of dimension at least five, then $\pi_0\Diff(M)$ is commensurable with an arithmetic group, hence is finitely presented and contains no non-trivial divisible elements.  The arithmetic group arises from automorphisms of the Sullivan minimal model; the forgetful map from $\pi_0\Diff(M)$ to the group of tangential self-homotopy equivalences of $M$ has finite kernel up to conjugacy \cite{Browder}. Arithmetic groups cannot contain non-trivial divisible elements. Indeed, their eigenvalues, in a fixed matrix representation, are algebraic integers of bounded degree. Therefore a divisible subgroup is composed of unipotent elements, and has Zariski closure a unipotent subgroup; but commutative unipotent groups only have free abelian arithmetic subgroups \cite{KMM}. 
(Note that general finitely presented groups \emph{can} contain non-trivial divisible elements: indeed, every finitely presented group is a subgroup of some finitely presented group $G$ whose commutator subgroup $[G,G]$ has only divisible elements and with $G/[G,G]\cong\bZ$, cf. \cite{Baumslag-MillerIII}.)  The same conclusions, in particular finite presentability, also apply to the Torelli group.  

Away from the simply connected case, much less is known.  For $n\geq 6$, the mapping class group $\pi_0\Diff(T^n)$ has a split subgroup $(\bZ/2)^{\infty}$ arising from the Whitehead group $\mathrm{Wh}_2$ of the fundamental group \cite{Hatcher:concordance, Hsiang-Sharpe}, so finite generation can fail. (This is in fact a phenomenon about homeomorphisms, rather distinct from questions of exotic smooth structures.)  The group of homotopy self-equivalences $hAut(M)$ when $\pi_1(M)$ is general also seems mysterious: can $hAut(M)$, or $\pi_0\Diff(M)$, ever contain a divisible group, for instance the rationals, as a subgroup?  Given the lack of complete computations, nothing seems to be known about the realisability problem, i.e. which groups occur as smooth mapping class groups: is there an $M$ for which $\pi_0\Diff(M)$ is a rank 2 free group?

\subsection{Symplectic manifolds}

Fix a symplectic $2n$-manifold $(X,\omega)$, closed or convex at infinity.  Let $\Symp(X,\omega)$ denote the group of diffeomorphisms of $X$ preserving $\omega$, equipped with the $C^{\infty}$-topology. 
\begin{enumerate}
\item If $H^1(X;\bR)=0$ then the connected component of the identity $\Symp_0(X) = \Symp(X,\omega) \cap \Diff_0(X)$  is exactly the subgroup $\Ham(X)$ of Hamiltonian symplectomorphisms, and the symplectic mapping class group is the quotient $\pi_0\Symp(X) = \Symp(X,\omega)/\Ham(X)$.
\item If $H^1(X;\bR)\neq 0$ then there is a flux homomorphism (defined on the universal cover)
\[
\Phi: \widetilde{\Symp}(X,\omega) \rightarrow H^1(X;\bR)
\]
whose kernel is the universal cover of the Hamiltonian group. The image $\Gamma$ of $\Phi$ viewed as a homomorphism on $\pi_1\Symp(X,\omega)$  is a discrete group \cite{Ono:flux}, and $\Symp_0(X)/\Ham(X) \cong H^1(X;\bR)/\Gamma$. $\Symp_0(X,\omega)$ carries a foliation by copies of $\Ham(X)$, whose leaf space is locally isomorphic to $H^1(X;\bR)$. Since Floer theory is Hamiltonian invariant,  it is sometimes useful to equip $\Symp(X,\omega)$ with the weaker ``Hamiltonian topology", in which (by definition) only isotopies along the leaves are continuous. 
\end{enumerate}
There is a forgetful map 
\[
q: \pi_0\Symp(X,\omega) \to \pi_0\Diff(X)
\]
which is typically neither injective nor surjective.  For non-injectivity, one has Seidel's results on squared Dehn twists and their inheritors \cite{Seidel:knotted, Tonkonog}: if $L\subset X$ is a Lagrangian sphere, there is an associated twist $\tau_L$, which has finite order smoothly if $n=\dim_{\bR}(L)$ is even, but typically has infinite order symplectically, for instance for smooth hypersurfaces of degree $>2$ in projective space.  If $n=2$, the smooth finiteness can be seen rather explicitly.  The cotangent bundle $T^*S^2 =\{x^2+y^2+z^2=1\}\subset \bC^3$ is an affine quadric; the Dehn twist $\tau_L$ in the zero-section $L$ is the monodromy of the Lefschetz degeneration $\{x^2+y^2+z^2=t\}$ of that quadric around the unit circle in the $t$-plane; that family base-changes to give a family with a 3-fold node  $\{x^2+y^2+z^2=t^2\}$; and the node admits a small resolution, so the base-change family has trivial smooth monodromy and $\tau_L^2=\id$ in (compactly supported) $C^{\infty}$.   

For non-surjectivity, a diffeomorphism cannot be isotopic to a symplectic diffeomorphism unless it preserves $[\omega]$ and $c_i(X) \in H^{2i}(X;\bZ)$, and furthermore permutes the classes with non-trivial Gromov-Witten invariants \cite{Ruan:extremal}. There are also diffeomorphisms acting trivially on cohomology which do not preserve the homotopy class of any almost complex structure \cite{ORW}, so $q$ is not onto the Torelli subgroup for many projective hypersurfaces.

There should be further, less homotopy-theoretic constraints on Torelli symplectomorphisms. A diffeomorphism $\phi: X\to X$ induces an automorphism of the $A_{\infty}$-algebra $H^*(X;\bC)$ (where the $A_{\infty}$-structure is  the classical one on cohomology) which may have non-trivial higher-order terms even when the linear action is trivial. If $X$ is Fano, its quantum cohomology
\begin{equation} \label{eqn:splits}
QH^*(X;\bC) \cong \oplus_{\lambda} QH^*(X;\lambda) \qquad \lambda \in \mathrm{Spec}(\ast c_1(X))
\end{equation}
splits into generalised eigenspaces for quantum product by $c_1(X)$, and the higher-order terms of any symplectomorphism should preserve this decomposition, i.e. be compatible with the quantum-corrected $A_{\infty}$-structure, a non-linear constraint.


\subsection{Monodromy}
If $X$ admits a K\"ahler metric $g$ with positive-dimensional isometry group $K \subset \mathrm{Isom}(X,g)$, then the map $K \to \Symp(X,\omega)$  is often rather interesting, at least for higher homotopy groups. Gromov  \cite{Gromov} proved that $\bP U(3) \simeq \Symp(\bP^2)$ and inferred that the group of compactly supported symplectomorphisms $\Symp_{ct}(\bC^2)$ is contractible;  for $n>2$ we don't know anything about $\pi_0\Symp_{ct}(\bC^n, \omega_{\st})$. It follows that fully computing $\pi_0\Symp(X,\omega)$ will not be feasible in essentially any higher-dimensional case, and replacing it with an algebraic avatar makes sense.

The most obvious source of information regarding $\pi_0\Symp(X,\omega)$ comes from algebraic geometry: if $X$ is a smooth projective variety and varies in a moduli space $\scrM$ (of complex structures, or complete intersections, with fixed polarisation), parallel transport yields a map 
\begin{equation} \label{eqn:monodromy}
\rho: \pi_1(\scrM) \to \pi_0\Symp(X,\omega);
\end{equation}
most known constructions of non-trivial symplectic mapping classes arise this way.  
Even in simple cases, classical topology struggles to say much about \eqref{eqn:monodromy}.  Let $p$ be a degree $m$ polynomial with distinct roots.  Let $X$ be the Milnor fibre $\{x^2+y^2+p(z)=0\}\subset \bC^3$ of the $A_{m-1}$-singularity $\bC^2/(\bZ/m)$. Parallel transport for the family over configuration space varying $p$, and simultaneous resolution of 3-fold nodes, now yields a diagram
\begin{equation} \label{eqn:diagram}
\xymatrix{
Br_m \ar[rr]^{\rho} \ar[d] && \pi_0\Symp_{ct}(X) \ar[d]\\
\Sym_m \ar[rr]^{\mu} && \pi_0\Diff_{ct}(X) 
}
\end{equation}
but  $\rho$ is actually injective \cite{Khovanov-Seidel}. For Milnor fibres of most singularities, the kernel 
\[
\pi_0\Symp_{ct}(X,\omega) \to \pi_0\Diff_{ct}(X)
\]
 is large  \cite{Keating:free_twists}.  In another direction,  let $\scrM_{d,n} = Z_{d,n}/\bP \operatorname{GL}_{n+2}(\bC)$ denote the moduli space of degree $d$ hypersurfaces in $\bP^{n+1}$, with $Z_{d,n} \subset H^0(\bP^{n+1},\mathcal{O}(d))$ the discriminant complement.  Tommasi \cite{Tommasi} proved that $\tilde{H}^*(\scrM_{d,n};\bQ) = 0$ for $d\geq 3$ and $\ast < (d+1)/2$, so the natural map $\scrM_{d,n} \to B\Diff(X_{d,n})$ cannot be probed using the rather rich cohomology of the image \cite{GRW}.   It seems unclear when \eqref{eqn:monodromy}
induces an interesting map on rational cohomology.

 Whilst (the possible failure of) \emph{injectivity} of $\rho$ can be probed using Floer-theoretic methods, these seem much less well-adapted to understanding the  possible \emph{surjectivity} of $\rho$; for questions of finite generation, or residual finiteness, one needs maps out of $\pi_0\Symp(X,\omega)$, i.e. one needs to have it act on something.  Whilst one can formally write down analogues of the complex of curves, the lack of classification results for Lagrangian submanifolds makes it hard to extract useful information.  Spaces of stability conditions on the Fukaya category emerge as another contender, simply because they have been effective in some situations on the other side of the mirror.
 
 Obviously, since $\pi_0\Symp(X,\omega)$ acts on spaces associated to the Fukaya category through the action of the quotient $\Auteq(\scrF(X,\omega))/\langle [2]\rangle$ (dividing by twice the shift functor), one can only hope to extract information about the latter; this is somewhat similar to first steps in surgery theory, in which one obtains information about $\pi_0\Diff(M)$ from the quotient group $hAut(M)$ of homotopy self-equivalences on the one-hand, as probed by surgery theory and $L$-theory,  and separately (and with separate techniques) from pseudo-isotopy theory and algebraic $K$-theory.  In classical surgery theory,  $B\Diff(M)$ is more accessible to study  (e.g. cohomologically)  than $\pi_0\Diff(M)$, and working at the space level is crucial for many applications. One can upgrade autoequivalences of an $A_{\infty}$-category to a simplicial space, but there are no compelling computations, and we will not pursue that direction.  We should at least mention the Seidel representation \cite{Seidel:QH_invertible} $\pi_1\Ham(X,\omega) \to QH^*(X)^{\times}$ (to the group of invertibles)
  as indication that the higher homotopy groups may carry interesting information away from the Calabi-Yau setting; when $c_1(X)=0$ the situation is less clear-cut.
 
We will say that a symplectic manifold $(X,\omega)$ is \emph{homologically mirror} to an algebraic variety $X^{\circ}$ over $\Lambda$ if there is a $\Lambda$-linear equivalence $D^{\pi}\scrF(X,\omega) \simeq D^b(X^{\circ})$.  The following result,  whilst  narrow in scope, gives a first indication that  knowing that something is a homological mirror might sometimes be leveraged to extract new information. 
 
 \begin{Proposition} Let $(X,\omega)$ be a $K3$ surface which is homologically mirror to an algebraic  K3 surface $X^{\circ}$. Then a symplectomorphism $f$ of $X$ which preserves the Lagrangian isotopy class of each  Lagrangian sphere in $X$ acts on $D^{\pi}\scrF(X)$ with finite order. If $f$ acts on $D^{\pi}\scrF(X)$ non-trivially then it acts on $H^*(X;\bC)$ non-trivially.
 \end{Proposition}
 
\begin{proof} An autoequivalence of $D^b(X^\circ)$ acting trivially on the set of spherical objects in $D^b(X^{\circ})$ arises from a transcendental automorphism of $X^{\circ}$, cf. \cite[Appendix A]{Huybrechts:spherical}.  The group of such is a finite cyclic group \cite[Ch.3 Corollary 3.4]{Huybrechts:K3book} acting faithfully on cohomology.  Finally, all spherical objects in $D^{\pi}\scrF(X)$ are quasi-isomorphic to Lagrangian sphere vanishing cycles \cite{SS:k3} (the proof given in \emph{op. cit.}, relying on constraining lattice self-embeddings via discriminant considerations, can be generalised away from the case when $X^{\circ}$ has Picard rank one).
\end{proof}
One could view this as a weak version of the fact that $\Gamma_g$ acts faithfully on the complex of curves.  Note that for a double plane $X \to \bP^2$ branched over a sextic one expects that every Lagrangian sphere arises from a degeneration of the branch locus, and the covering involution reverses orientation but preserves its (unoriented) Lagrangian isotopy class, suggesting the conclusion may be optimal. 
    
\section{Stability conditions}  

\subsection{Definitions\label{Subsec:Definitions}} Let $\scrC$ be a proper, i.e. cohomologically finite, triangulated category, linear over a field $k$.  We will assume that the numerical Grothendieck group $K(\scrC)$, i.e. the quotient of the Grothendieck group $K^0(\scrC)$ by the kernel of the Euler form,  is free and of finite rank $d$. The space of (locally finite numerical) stability conditions will then be a $d$-dimensional complex manifold.

 A \emph{stability condition} $\sigma=(Z,\P)$ on   $\scrC$
consists of
a group homomorphism
$Z\colon K(\scrC)\to\C$ called the \emph{central charge},
and full additive
subcategories $\P(\phi)\subset\scrC$ of \emph{$\sigma$-semistable objects of phase $\phi$}  for each $\phi\in\bR$, 
which together satisfy a collection of axioms, the most important being:
\begin{itemize}
\item[(a)] if $E\in \P(\phi)$ then $Z(E)\in \bR_{>0}\cdot e^{i\pi\phi}\subset \C$;
\item[(b)] for each nonzero object $E\in\scrC$ there is a finite sequence of real
numbers $\phi_1>\phi_2> \dots >\phi_k$ 
and a collection of triangles
\[
\xymatrix@C=.5em{
0_{\ } \ar@{=}[r] & E_0 \ar[rrrr] &&&& E_1 \ar[rrrr] \ar[dll] &&&& E_2
\ar[rr] \ar[dll] && \ldots \ar[rr] && E_{k-1}
\ar[rrrr] &&&& E_k \ar[dll] \ar@{=}[r] &  E_{\ } \\
&&& A_1 \ar@{-->}[ull] &&&& A_2 \ar@{-->}[ull] &&&&&&&& A_k \ar@{-->}[ull] 
}
\]
with $A_j\in\P(\phi_j)$ for all $j$.
\end{itemize}
We will always furthermore impose the ``support property": for a norm $\|\cdot\|$ on $K(\scrC)\tensor\bR$ there is a constant $C>0$ such that $\|\gamma\|< C\cdot |Z(\gamma)|$ 
for all $\gamma\in K(\scrC)$  represented by $\sigma$-semistable objects in $\scrC$.
(This in particular means all our stability conditions are ``full", cf. \cite[Proposition B.4]{Bayer-Macri-localP2}.) The semistable objects $A_j$ appearing in the filtration of axiom (b) are unique up to isomorphism, and are called the \emph{semistable factors} of $E$. We set
\[\phi^+(E)=\phi_1, \quad \phi^-(E)=\phi_k, \quad m(E)=\sum_{i} |Z(A_i)|\in \bR_{>0}.\]
A simple object $E$ of $\P(\phi)$ is said to be stable of phase $\phi$ and mass $|Z(E)|$.  
Let  $\Stab(\scrC)$ denote the set of all stability conditions on $\scrC$.
It carries a natural topology, induced by the metric
 \begin{equation}
 \label{dist}
d(\sigma_1,\sigma_2) = \sup_{0\neq E\in\scrC}\left\{
  |\phi_{\sigma_2}^- (E) - \phi_{\sigma_1}^-(E)|, \  |\phi_{\sigma_2}^+
  (E) - \phi_{\sigma_1}^+(E)|, \ \left|\log
    \frac{m_{\sigma_2}(E)}{m_{\sigma_1}(E)}\right|  
\right\}\in[0,\infty].
\end{equation}

\begin{thm}[Bridgeland, cf. \cite{Bridgeland}]
\label{basic}
The space $\Stab(\scrC)$ has the structure of a complex manifold, such that the forgetful map
$\pi\colon \Stab(\scrC)\lra \Hom_{\Z}(K(\scrC),\C)$
taking a stability condition to its central charge is a local isomorphism.
\end{thm}

The group of triangulated autoequivalences $\Auteq(\scrC)$ acts on $\Stab(\scrC)$ by
holomorphic automorphisms which preserve the metric. 
There is a commuting (continuous, non-holomorphic) action of  the universal cover of the group  $\operatorname{GL^+}(2,\bR)$, not changing the subcategories $\P$, but acting by post-composition on the central charge viewed as a map to $\C=\bR^2$ (and correspondingly adjusting the $\phi$-labelling of $\P$). A  subgroup $\C \subset \widetilde{\operatorname{GL}}^+(2,\bR)$ acts freely, by
\[\bC \ni t\colon (Z,P)\mapsto (Z',\P'), \quad Z'(E)=e^{-i\pi t}\cdot Z(E), \quad \P'(\phi)=\P(\phi+\operatorname{Re}(t)).\]
For any  integer $n$, the action of the shift $[n]$ coincides with the action of $n\in \C$.   
Of particular relevance is the quotient $\Stab(\scrC)/\langle [2]\rangle$, on which the $\widetilde{\operatorname{GL}}^+(2,\bR)$-action descends to a $\operatorname{GL}^+(2,\bR)$-action;  in the symplectic setting, this will amount to focussing attention on symplectomorphisms rather than graded symplectomorphisms. We will write $\Stab(\scrF(X))$ for $\Stab(D^{\pi}\scrF(X,\omega))$.

\begin{Remark} If a power of the shift functor of $\scrC$ is isomorphic to the identity,  $\Stab(\scrC) = \emptyset$, so $\Stab(\scrF(X))$ is only non-trivial when $2c_1(X)=0$.   Phantom subcategories of derived categories of coherent sheaves \cite{Gorchinskiy-Orlov:phantom} have trivial $K$-theory so also cannot admit any stability condition. \end{Remark}

The definition was motivated by ideas in string theory which in turn connect closely to symplectic topology: if $(X,\omega)$ is a symplectic manifold with $c_1(X)=0$,   there is conjecturally an injection from the Teichm\"uller space $\mathcal{M}_X(J,\Omega)$ of marked pairs $(J,\Omega)$ comprising a compatible integrable complex structure $J$ and $J$-holomorphic volume form $\Omega \in \Omega^{n,0}(X;J)$ into $\Stab(\scrF(X))$.   In this scenario, given $(J,\Omega)$, then $Z(L) = \int_L\Omega$, the categories $\P(\phi)$ contain the special Lagrangians of phase $\phi$, and the Harder-Narasimhan filtration should be the output of some version of mean curvature flow with surgeries at finite-time singularities, cf. \cite{Joyce:stability_conjectures}. 

An important point is that even if one can build a map $\mathcal{M}_X(J,\Omega) \to \Stab(\scrF(X))$,  simply for dimension reasons it often can't be onto an open subset; one expects it to have image a complicated transcendental submanifold of high codimension. There is rarely a predicted geometric interpretation of the ``general" stability condition.  (Completeness of \eqref{dist} was proven in \cite{Woolf:complete}; contrast with the Weil-Petersson metric on moduli spaces of Calabi-Yau's, which frequently has the boundary at finite distance \cite{Wang} and is incomplete.) More positively, since the vanishing cycle of any nodal degeneration can be realised by a special Lagrangian \cite{Hein-Sun}, one expects any Lagrangian sphere which is such a vanishing cycle to be stable for some stability condition, and to define an ``end" to the space of stability conditions, where the mass of this stable object tends to zero.  Thus one might hope that the global topology of $\Stab(\scrF(X))/ \Auteq(\scrF(X))$ carries  information about Lagrangian spheres.

\subsection{Properties}

$\Stab(\scrC)$ is a complex manifold. But it inherits additional structure: it is modelled on a fixed vector space such that the transition maps between charts are locally the identity, and it has a global \'etale map to $\Hom_{\Z}(K(\scrC),\bC) \cong \bC^d$.

Fix a connected component $\Stab_{\dagger}(\scrC)$ of $\Stab(\scrC)$.  Let $\Auteq_{\dagger}(\scrC)$ denote the quotient of the subgroup of $\Auteq(\scrC)$ which preserves $\Stab_{\dagger}(\scrC)$ by the subgroup of ``negligible" autoequivalences, i.e. those which act trivially on $\Stab_{\dagger}(\scrC)$.  Negligible autoequivalences can exist, for instance coming from automorphisms of projective surfaces which act trivially on the algebraic cohomology (these need not be trivial on the transcendental cohomology).

\begin{Lemma} \label{Lem:features} $\Stab_{\dagger}(\scrC)$ has a well-defined integral affine structure and a canonical measure.  The quotient $\scrQ =  \Stab_{\dagger}(\scrC)/\Auteq_{\dagger}(\scrC)$ is a complex orbifold, with $2c_i(\scrQ)=0$ for $i$ odd.
\end{Lemma}

\begin{proof} The affine structure comes from the integral lattice in $K(\scrC)$. The measure is obtained from Lebesgue measure in $\Hom_{\Z}(K(\scrC),\bC)$, normalised so that the quotient torus $\Hom_{\Z}(K(\scrC), \bC/(\bZ \oplus i\bZ))$ has volume one.    We claim that, having killed the generic stabiliser, $\Auteq_{\dagger}(\scrC)$ acts with finite stabilisers. The argument is due to Bridgeland, and applies to the space of stability conditions satisfying the support condition whenever the central charge is required to factor through a finite rank lattice $N$ (see also \cite{Huybrechts:conway}). Given any stability condition $\sigma$ and $R\in \bR_{>0}$, there are only finitely many classes $n\in N$  which are represented by a $\sigma$-semistable object $E$ with $|Z_{\sigma}(E)|<R$.  For $R\gg 0$ sufficiently large, these classes will span the lattice (fix a set of objects of $\scrC$ whose classes give a basis of $N$, and consider the classes represented by the semistable factors in the HN-filtrations of those objects).  Any autoequivalence $\phi$ fixing $\sigma$ must permute this finite set, so some fixed power $\sigma^k$ fixes all these elements pointwise. Then $\sigma^k$ acts trivially on $\Hom_{\bZ}(N, \bC)$, hence by Theorem \ref{basic} on a neighbourhood of $\sigma \in \Stab_\dagger(\scrC)$, and hence acts trivially globally (since it's a holomorphic automorphism). This shows that for any $\sigma \in \Stab_{\dagger}(\scrC)$, the stabiliser of $\sigma$ acts faithfully on a finite set.  The proof also shows that the stabiliser of a point $\sigma \in \Stab_{\dagger}(\scrC)$ injects into $\operatorname{GL}(d;\bZ)$ via the action on the central charge, so the  torsion orders of stabilisers are bounded only in terms of $d$.

To deduce that the quotient is an orbifold, it remains to see that the action of $\Auteq_{\dagger}(\scrC)$ is properly discontinuous, so admits slices.  \cite[Lemma 6.4]{Bridgeland} proves that if two stability conditions $\sigma, \tau$ have the same central charge, then they co-incide or are at distance $d(\sigma,\tau) \geq 1$. Therefore, if the action was not properly discontinuous, there would be a stability condition $\sigma$ and an infinite sequence of elements in $\operatorname{GL}(d,\bZ)$ which failed to displace a small ball around $Z_\sigma$.  This easily yields a contradiction.  Since autoequivalences act on $\Hom_{\bZ}(K(\scrC),\bC)$ via the complexifications of integral linear maps, they preserve the real and imaginary subbundles of the tangent bundle. Thus,  $\Stab_{\dagger}(\scrC)$ has holomorphically trivial tangent bundle, and the tangent bundle of $\Stab_{\dagger}(\scrC)/\Auteq_{\dagger}(\scrC)$ is the complexification of a real bundle, so the odd Chern classes of the quotient orbifold are 2-torsion.
\end{proof}




The diagonal group $\diag(e^t,e^{-t}) \subset \operatorname{SL}(2,\bR)$ acts on $\Stab(\scrC)$ by expanding and contracting the real and imaginary directions in $\Hom_{\bZ}(K(\scrC),\bC)$, which are tangent to smooth Lagrangian subbundles of the tangent bundle (for the flat K\"ahler form on $\bC^d$), somewhat reminiscent of an Anosov flow.  If this flow on $\Stab(\scrC)$ was the geodesic flow of a complete Riemannian metric, then classical results (non-existence of conjugate points) would imply that $\Stab(\scrC)$ was  a $K(\pi,1)$.  Despite the naivety of such reasoning, in known cases, the connected components of $\Stab(\scrC)$ are contractible, or diffeomorphic to spaces independently conjectured to be contractible in the literature \cite{Allcock:coxeter, Kontsevich-Zorich, Qiu-Woolf}.

\begin{Corollary}
If $\Stab_{\dagger}(\scrC)$ is contractible,  $\Auteq_{\dagger}(\scrC)$ has finite rational cohomological dimension.
\end{Corollary}

\begin{proof} This is true for any discrete group $\Gamma$ acting on a contractible finite-dimensional manifold $Q$ with finite stabilisers, provided the action admits slices.  (To prove the result, apply the Leray-Hirsch theorem to the projection $(Q \times E\Gamma) \to Q \times_{\Gamma} E\Gamma$, and note that the stalks of the pushforward of the constant sheaf $\bQ$ vanish.)  The existence of slices for the action follows from proper discontinuity.
\end{proof}

If a finitely generated countable group $G$ has an embedding $G\times G \hookrightarrow G$ (e.g. the finitely presented ``Thompson's group $F$'') then it has infinite rational cohomological dimension \cite{Gandini}, so the contractibility hypothesis already excludes some ``reasonable" countable groups from being autoequivalence groups.


 It is tempting to believe contractibility holds, when it does, for some intrinsic geometric reason, e.g. that the canonical metric $\eqref{dist}$ is non-positively curved, making $\Stab(\scrC)$ a complete CAT(0) space.  (If the phases of $\sigma$-semistable objects are dense in $S^1$, then the $\operatorname{GL^+}(2,\bR)$-orbit of $\sigma$ is free, and the metric on the quotient $\operatorname{GL^+}(2,\bR) / \bC$ is the standard hyperbolic metric on the upper half-plane up to scale \cite{Woolf:complete}.) 
This would be of dynamical relevance. Groups acting on weakly hyperbolic spaces with rank one elements (which typically exist) admit infinite-dimensional families of  quasimorphisms, for instance the ``counting" quasimorphisms of \cite{Epstein-Fujiwara}, see also \cite{Calegari:scl}.  On the other hand, \cite{Malyutin} shows that if $\Phi: G \to \bR^3$ is defined by a triple of unbounded quasimorphisms, and if $S \subset G$ has bounded $\Phi$-image, then $S$ is transient, meaning that  a random walk on $G$ (defined with respect to any non-degenerate probability measure, i.e. one whose support generates $G$ as a semigroup) will visit $S$ only finitely many times almost surely.  A prototypical example is that the reducible or periodic surface diffeomorphisms --  those for which $\iota(\alpha,\phi^n(\beta))$ is not exponential for all $\alpha,\beta$ in (\ref{Subsec:surfaces}, (3)) -- are transient in $\Gamma_g$, and random mapping classes of surfaces are pseudo-Anosov.
   
 \begin{Corollary}\label{Cor:scl}
 If $\Stab_{\dagger}(\scrC)$ is complete CAT(0) and $\Auteq_{\dagger}(\scrC)$ acts with rank one elements, then $scl$ of the $k$th element of a random walk tends to infinity as $k$ tends to infinity almost surely. 
 \end{Corollary}

 One can analogously look for quantitative ``unbounded generation" results. \cite{BHC} proved the isometry group of an arithmetic lattice acts on elements of fixed square with finitely many orbits.  Suppose this homological statement lifts, and that $\scrC$ has only finitely many conjugacy classes of spherical object under autoequivalence.  Quasimorphisms can be averaged to be homogeneous, hence constant on conjugacy classes, so then all spherical twists have uniformly bounded image under any $\Phi$ as above. If $S\subset G$ is transient and $N\subset G$ is finite, then $(S\cup S^{-1} \cup N)^k$ is transient for any fixed $k$. 
 For $\scrC=\scrF(X)$ we have noted that one hopes a Lagrangian sphere defines an end to (a component meeting $\mathcal{M}_X(J,\Omega)$ of)  $\Stab(\scrF(X))$ corresponding to a nodal degeneration.  Suppose that there is a partial  compactification of $\scrQ=\Stab_{\dagger}(\scrF(X))/\Auteq_{\dagger}(X)$, or of $\mathcal{M}_X(J,\Omega)$,  with an irreducible primitive analytic divisor $D$ parametrizing nodal degenerations of $X$.  A choice of  a closed oriented surface $S$  transverse to $D$ at one point exhibits, via the usual presentation for  $\pi_1(S)$, the corresponding spherical twist as a product of commutators, and hence all conjugate spherical twists as products of the same number of commutators.  Then Corollary \ref{Cor:scl} would imply that, for any fixed $k$, a random element of $\Auteq_{\dagger}(\scrF(X))$ would almost surely not be a product of fewer than $k$ spherical twists. 
 
 This raises the question: how does one find \emph{Floer-theoretic} conditions for  sets of symplectomorphisms to lie in a bounded set under the image of a quasimorphism on $\pi_0\Symp(X)$?

\subsection{Classification of objects\label{Subsec:examples}}  Consider some concrete computations. Since the space of stability conditions is expected to be contractible, the object of interest is the  orbifold covering map 
\begin{equation} \label{eqn:cover}
\pi: \Stab_{\dagger}(\scrC) \to \Stab_{\dagger}(\scrC)/\Auteq_{\dagger}(\scrC) = \scrQ.
\end{equation}

\begin{itemize}
\item $(X,\omega)$ is $(T^2,\omega_{\st})$.  Then $\scrQ$ is a $\bC^*$-bundle over $\frak{h}/ SL(2,\bZ)$, and $\Auteq(\scrF(X))$ is an extension of $SL(2;\bZ)$ by $\bZ \oplus X \times X^{\vee}$ (acting by shift, translation and tensoring by flat bundles); \cite{Bridgeland, Polishchuk-Zaslow}. 

\item $(X,\omega)$ is $\{x^2+y^2+z^k=1\}\subset (\bC^3, \omega_{\bC^3})$ and we consider the compact Fukaya category. Then $\scrQ = \Conf_k(\bC)$  parametrises polynomials with $k$ distinct roots, cf. the monodromy discussion of \eqref{eqn:diagram}; $\Auteq(\scrF(X))/[2] = Br_k$ is the braid group \cite{Thomas:braid, Ishii-Ueda-Uehara}.

\item $(X,\omega)$ is $(\Sigma, M)$, a surface with non-empty boundary containing a non-empty set of boundary marked points $M\subset \partial \Sigma$, and we consider the partially wrapped\footnote{The compact Fukaya category $\scrF$ is proper but not smooth; the wrapped category $\scrW$ is smooth but not proper; the partially wrapped category, which depends on additional data -- here, the set $M$ -- has both properties. It determines $\scrW$ as a localization, and (conjecturally, in nice cases) $\scrF$ as a subcategory of compact objects of $\scrW$.}  Fukaya category $\scrF(\Sigma;M)$ stopped at $M$ \cite{Sylvan}, whose objects include arcs ending on $\partial\Sigma\backslash M$. Then $\scrQ$ is a space of marked flat structures (meromorphic quadratic differentials) on $\Sigma$, up to diffeomorphism \cite{HKK}.

\end{itemize}
The computation of $\Stab(\scrC)$, when feasible, often goes hand-in-hand with a classification for some class of objects of the category $\scrC$, which is of independent interest.  (Whilst the objects of $\scrF(X)$ are \emph{a priori} geometric, the objects of $\scrF(X)^{per\!f}$ are not, and a concrete  interpretation of an arbitrary perfect module is rarely available.) In the cases above, one uses:

\begin{itemize}
\item Atiyah's classification of bundles on elliptic curves, which implies that a twisted complex on objects in $\scrF(T^2)$ is again quasi-represented by a simple closed curve with local system; 

\item Ishii-Uehara's proof \cite{Ishii-Uehara} that all spherical objects in $\Tw\scrF(X_p)$ are in the braid group orbit of a fixed Lagrangian sphere; 

\item Haiden-Katzarkov-Kontsevich's proof  that every indecomposable twisted complex of objects of $\scrF(\Sigma;M)$ is represented by an immersed curve with local system.

\end{itemize}
 By appealing to either sheaf-theoretic properties of wrapped categories \cite{Lee} or equivariant arguments \cite{Seidel:equivariant_c*}, one can deduce from the last example above that on a closed surface $\Sigma_g$, every spherical object is geometric, i.e. quasi-isomorphic to an immersed curve with local system. It is tempting to speculate that, starting from that point, one can recover the complex of curves and hence the classical mapping class group from the derived Fukaya category $D^{\pi}\scrF(\Sigma_g)$.

\subsection{Measure and growth}

Recall that $\Stab_\dagger(\scrC)$ has a canonical measure $d\nu$ inherited from Lebesgue measure in $\Hom_{\bZ}(K(\scrC), \bC)$.  Since autoequivalences act linearly and integrally on $K(\scrC)$, the measure is $\Auteq_{\dagger}(\scrC)$-invariant, so descends to the quotient orbifold.  This will essentially always have infinite measure, since there is a free $\bR$-action on $\Stab(\scrC)$,  rescaling  the central charge.

\begin{Hypothesis} \label{Hyp:finite_measure} There is a submanifold $\Stab_{=1,\dagger}(\scrC)$ which is a slice to the $\bR$-action, which is invariant under the universal cover of $SL(2,\bR)$, and for which $\Stab_{=1, \dagger}(\scrC) / \Auteq_{\dagger}(\scrC)$ has finite measure. 
\end{Hypothesis}

Existence of a slice is usually not difficult, but the finite measure hypothesis is probably much more severe. Let $S$ be a Riemann surface and $\phi \in H^0(2K_S)$ a quadratic differential with distinct zeroes.  There is a local CY K\"ahler threefold
\begin{equation} \label{eqn:local_holomorphic}
Y_{\phi} = \{(q_1,q_2,q_3)\in K_S \oplus K_S \oplus K_S \, | \, q_1^2+q_2^2+q_3^2 = \phi\}
\end{equation}
which depends up to deformation only on the genus $g(S)$. The space of stability conditions  modulo autoequivalences on $\scrF(Y_{\phi})$  is conjecturally the moduli space $\scrM(S,\phi)$ of quadratic differentials with simple zeroes, and the normalisation $\int_S \phi \wedge \bar{\phi} = 1$ to surfaces of flat area 1 defines a slice to the $\bR$-action. This has finite measure by deep results of Veech \cite{Veech:flat_surfaces}.  For a K3 surface $X^{\circ}$, $Z(E) = (\Omega,ch(E))$ for a period vector $\Omega \in \bN(X^{\circ})\otimes \bC$, where $\bN(X^{\circ})$ is the extended Neron-Severi group $H^0\oplus \Pic \oplus H^4$, and the condition $(\Omega,\bar{\Omega}) =1$ defines a slice. 

Given \ref{Hyp:finite_measure}, one can play the following game (a well-known trope in the flat surfaces community; we follow  \cite{Zorich:flat_surfaces}).  Fix a stability condition $\sigma$; the support property means that the central charges $Z(E)$ of $\sigma$-semistable  objects are discrete in $\bC$. For a compactly supported function $f\in C_{ct}(\bR^2)$ we can then define $\hat{f}(\sigma) = \sum_{Z(E)} f(Z(E))$, summing over central charges represented by semistable $E$.  A stronger version of Hypothesis \ref{Hyp:finite_measure} asks that the function $\sigma \mapsto \hat{f}(\sigma)$ be $d\nu$-integrable, yielding a linear function
\[
f \mapsto \int_{\Stab_{=1,\dagger}(\scrC)/\Auteq_{\dagger}(\scrC)} \hat{f} \, d\nu
\]
which is $SL_2(\bR)$-invariant. 
The only such functionals on $C_{ct}(\bR^2)$ are given by the total area and the value at zero; the latter is irrelevant here since $Z(E)\neq 0$ for any semistable $E$, so one finds that 
\[
\mathrm{For \ every} \ f\in C_{ct}(\bR^2), \ \int_{\Stab_{=1,\dagger}(\scrC)/\Auteq_{\dagger}(\scrC)} \hat{f} \, d\nu = \tau(\scrC) \int_{\bR^2} f \, dxdy
\]
where $\tau(\scrC)$ is an invariant of the category and the choice of component / slice. Taking $f$ to be the indicator function of a disc of increasing radius $R$, one sees that $\tau(\scrC)$ measures the growth rate of $K$-theory classes represented by semi-stable objects as their mass increases. For $\scrM(S,\phi)$ above, $\tau(\scrC)$ is a  Siegel-Veech number, and controls the quadratic growth rate of special Lagrangian $3$-spheres by volume in the 3-fold \eqref{eqn:local_holomorphic} (in that case, the special Lagrangian $3$-sphere representative of the homology class is expected to be unique).

\section{Three more extended examples\label{Sec:Examples}} 

\subsection{K3 surfaces}
Let $Y$ be an algebraic K3 surface of Picard rank $\rho \geq 1$. Bridgeland \cite{Bridgeland-K3} has given a conjectural description of a distinguished component of $\Stab(D^b(Y))$ as the universal cover of a period domain.  Recall $\bN(Y) = H^0(Y)\oplus \Pic(Y) \oplus H^4(Y)$, which carries the Mukai pairing $\langle (a,b,c),(a',b',c')\rangle = ac'+ca' - b\cdot b'$. Let 
\[
\Omega = \{u \in \bN(Y)\otimes \bC \, | \, \mathrm{Re}(u), \mathrm{Im}(u) \ \mathrm{span \ a \ positive \ definite \ 2{-}plane}\}.
\] 
Fix one of the two connected components of $\Omega$, and remove the locally finite union of hyperplanes $\delta^{\perp}$ indexed by $-2$-vectors  $\delta \in \bN(Y)$ (with its Mukai form), to obtain a period domain $\scrP^+_0(Y)$. Then \cite{Bridgeland-K3} constructs a connected component $\Stab_\dagger \subset \Stab(D^b(Y))$ and proves that there is a Galois covering
\[
\Stab_\dagger \to \scrP^+_0(Y)
\]
which identifies the Torelli autoequivalence group with the group of deck transformations.  Bridgeland conjectures that $\Stab_\dagger$ is simply-connected and $\Auteq$-invariant, which would yield the exact sequence
\[
1 \to \pi_1(\scrP^+_0(Y)) \to \Auteq_{CY}(D^b(Y)) \to \Aut^+(H^*(Y)) \to 1
\]
where the Calabi-Yau autoequivalences $\Auteq_{CY}$ are those acting trivially on the Hochschild homology group $HH_2(Y)$, and the final term is the index two subgroup of the full Hodge isometry group of automorphisms preserving the orientation of a maximal positive definite subspace. The generators of the first term would be mapped to squared spherical twists. 
Separately, a conjecture of Allcock \cite{Allcock:coxeter} on Coxeter arrangements would imply that, if Bridgeland's conjecture holds,  then $\Stab_{\dagger}(D^b(Y))$ is indeed  a CAT(0) space, in particular contractible. When $\rho(Y)=1$, both conjectures are known, which yields a complete determination of both that component of the space of stability conditions and the autoequivalence group.

\begin{Proposition}
Let $(X,\omega)$ be a K\"ahler K3 surface which is homologically mirror to an algebraic K3 surface $Y$ of Picard rank one.   (Minimality of the Picard rank on the mirror corresponds to an irrationality hypothesis on the K\"ahler form $\omega$ on $X$.) Let $G\scrF(X) = \Auteq(D^{\pi}\scrF(X,\omega))$. 
\begin{enumerate}
\item  $G\scrF(X)$ is finitely presented. 
\item The Torelli subgroup is infinitely generated and torsion-free.
\item  $G\scrF(X)$ contains no divisible elements.
\item $G\scrF(X)$ cannot obey a law.
\end{enumerate}
\end{Proposition}

\begin{proof}[Sketch comments]
The first three statements are essentially immediate from \cite{Bayer-Bridgeland}; in fact the categorical Torelli group in the (mirror to a) Picard rank one case is a countably generated free group.   $scl$ for free products of cyclic groups was computed in \cite{Walker}, e.g. there are examples for which $G\scrF(X) = \bZ/p \ast \bZ$, and for $\bZ/p \ast \bZ$ with generators $a,b$ of the factors, $scl([a,b]) = 1/2-1/p$; as remarked previously, non-vanishing of $scl$ rules out laws. See \cite{SS:k3} for related ideas and details.
\end{proof}

Proofs of the  Bridgeland and Allcock conjectures would eliminate the Picard rank one hypothesis.  Results for $G\scrF(X)$ do not immediately yield results for the symplectic mapping class group $\pi_0\Symp(X,\omega)$, but these can sometimes be extracted with more care. Let $(X,\omega)$ be the mirror quartic, a crepant resolution of $\{\sum_{j=0}^3 x_j^4 + \lambda \prod_j x_j = 0\} / \Gamma$ for $\Gamma \cong (\bZ/4)^{\oplus 2}$, equipped with an ``irrational" toric K\"ahler form (one induced from its embedding into a toric resolution of $[\bP^3/\Gamma]$ such that the areas of the resolution curves and of a hyperplane section are rationally independent). Then (see \cite{SS:k3}):

\begin{Theorem}[Sheridan, Smith] \label{Thm:SS}   For the mirror quartic with an irrational toric K\"ahler form, the group  $\ker(\pi_0\Symp(X)\to\pi_0\Diff(X))$ is infinitely generated. \end{Theorem}

The map $\pi$ of \eqref{eqn:cover} is (modulo $GL(2,\bR)$) the universal cover of the orbifold $\frak{h} / \Gamma_0(2)^+$, where $\Gamma_0(2)^+ = \bZ/2 \ast \bZ/4$, and $\Auteq_{CY}(\scrF(X))/[2] = \bZ\ast\bZ/4$ where again the CY autoequivalencs are those acting trivially on $HH_2(\scrF(X))$, cf. \cite{Bayer-Bridgeland, SS:k3}. One infers the Dehn twist $\tau_L$ in a Lagrangian sphere $L\subset X$  admits no non-trivial root in $\pi_0\Symp(X)$ (being a smooth involution, it is its own $(2p+1)$-st root in $\pi_0\Diff(X)$, for any $p\geq 1$); $\pi_0\Symp(X)$ is not generated by torsion elements (which all map to the same factor in the abelianization of $\Auteq(\scrF(X))$), in contrast to $\Gamma_g$, etc.  

\begin{Remark}
It is interesting to compare the proof in \cite{Bayer-Bridgeland} with the cartoon description of stability conditions in terms of special Lagrangians given at the end of  Section \ref{Subsec:Definitions}.  For a Picard rank one K3 surface $X^{\circ}$, the central charge is given by $Z(E) = \langle \Omega,ch(E)\rangle$ for a vector $\Omega \in \bN(X^{\circ})\otimes \bC \cong \bC^3$ whose real and imaginary parts span a positive definite two-plane. There is a unique negative definite vector $\theta \in \bN(X^{\circ})\otimes \bR$ orthogonal to $\{Re(\Omega),Im(\Omega)\}$. Fix $E=\mathcal{O}_x$ the skyscraper sheaf of a point,  and a stability condition $\sigma$ for which the largest and smallest semistable factors $A_{\pm}$ of $E$ have phases $\phi_+ > \phi_-$ respectively. Then \cite{Bayer-Bridgeland} studies the flow on $\Stab(X^{\circ})$ defined by
\[
d\Omega/dt = \xi.\theta \qquad \xi = i\exp(i\pi/2 (\phi_+ + \phi_-))
\]
which locally pushes $A_{\pm}$ towards one another, decreasing $\phi_+ - \phi_-$. They prove such flows can be patched together and eventually contract $\Stab(X^{\circ})$ to the geometric chamber where $E$ is semistable.  Translating back to the $A$-side, the cartoon is now that instead of mean curvature flow, one fixes the Lagrangian torus, and flows in the space of holomorphic volume forms to try to make it special. \end{Remark}


\subsection{Quiver threefolds}

Let $(Q,W)$ be a quiver with potential.  This determines  a 3-dimensional Calabi-Yau category $\scrC(Q,W)$ \cite{Ginzburg:cy}. If $(Q,W)$ has no loops, there is a ``mutation" operation which yields another $(Q',W')$ and a (pair of) derived equivalence(s) $\scrC(Q,W) \simeq \scrC(Q',W')$.  In a number of interesting cases, these categories are related to Fukaya categories of threefolds:
\begin{enumerate}
\item  The zero-potential on the two-cycle quiver (arrows labelled $e,f$)  is realised within the compact Fukaya category of the affine quartic $\{x^2+y^2+(zt)^2=1\}\subset \bC^4$, which is  a plumbing of two 3-spheres along a circle (plumbed so the Lagrange surgery is an $S^1\times S^2$), see \cite{EWS}. 
\item The potential $(ef)^2$ on the same quiver is realised by the compact Fukaya category of the  complement of a smooth hyperplane section in the variety of complete flags in $\bC^3$, which is again a plumbing of two 3-spheres along a circle (plumbed so the Lagrange surgery is an $S^3$);  potentials $(ef)^p$ on the two-cycle quiver, for prime $p>2$, arise in characteristic $p$ from the corresponding plumbings where the surgery is a  Lens space $L(p,1)$, see \cite{EWS}.
\item The potential associated \cite{Labardini-Fragoso} to an ideal triangulation of a marked bordered surface $(S,M)$ is realised by the Fukaya category of a threefold which is a conic fibration over $S$ with special fibres at $M$ \cite{Smith:quiver} (this is a cousin of the space from \eqref{eqn:local_holomorphic}, where $M=\emptyset$).
\end{enumerate}
The category $\scrC(Q,W)$ has a distinguished heart, equivalent to the category of nilpotent representations of the Jacobi algebra $\Jac(Q,W)$, with  $d$ simple objects up to isomorphism if $Q$ has $d$ vertices. This implies that a large subset  $\scrU \subset \Stab(\scrC(Q,W))$ with non-empty open interior is  a union of cells $\hbar^d$ (where $\hbar$ is the union of the upper half-plane and the negative real axis excluding zero), indexed by $t$-structures having hearts with finite length, glued together along their boundaries by the combinatorics of tilting (quiver mutation).  This provides one of the most direct routes to (partial) computations of spaces of stability conditions. Often, the image of $\scrU$ under the natural circle action on $\Stab(\scrC)/\langle [2]\rangle$ covers a path-component.

Let $\phi$ be a meromorphic quadratic differential on a surface $S$ with at least one pole of order $\geq 2$, with $p$ double poles, and distinct zeroes. Let $M\subset S$ be the set of poles. There is a threefold $Y_{\phi} \to S$, a variant of that from \eqref{eqn:local_holomorphic}, now with empty fibres over poles of order $>2$ and reducible fibres (singular at infinity) over double poles; a choice of component of each reducible fibre defines a class $\eta \in H^2(Y_{\phi};\bZ/2)$. Let $\scrF(Y_{\phi};\eta)$ denote the subcategory of the $\eta$-sign-twisted Fukaya category split-generated by Lagrangian spheres. Then (see \cite{BS, Smith:quiver}):

\begin{Theorem}[Bridgeland, Smith]  \label{Thm:BS} There is an equivalence $\scrF(Y_{\phi};\eta) \simeq \scrC(Q,W)$ for $(Q,W)$ the quiver with potential associated to any ideal triangulation of $(S,M)$.  Moreover, 
\begin{equation} \label{eqn:auteq_quiver_threefold}
1 \to \mathrm{Sph}(S,M) \to \Auteq(\scrC(Q,W)) \to \Gamma^{\pm}(S,M) \to 1
\end{equation}
where $\Gamma^{\pm}$ is an extension of the mapping class group $\Gamma(S,M)$ by $(\bZ/2)^p$.  
\end{Theorem}
The first factor in \eqref{eqn:auteq_quiver_threefold} acts through spherical twists and admits a natural representation to $\pi_0\Symp_{ct}(Y_{\phi})$ whilst the quotient factor acts via non-compactly-supported elements of $\pi_0\Symp(Y_{\phi})$.   In particular, the natural map $\mathrm{Sph}(S,M) \to \pi_0\Symp_{ct}(Y_{\phi})$ is actually split. Stable objects in $Y_{\phi}$ are all given by special Lagrangian 3-spheres or $S^1\times S^2$'s, corresponding to open and closed saddle connections in the flat metric on $(S,\phi)$. 

There are  3-folds $Y_{\phi,\psi}$ associated to a pair $(\phi,\psi) \in H^0(K_S^{\oplus 2}) \oplus H^0(K_S^{\oplus 3})$, fibred over $S$ by $A_2$-Milnor fibres rather than $A_1$-Milnor fibres: given line bundles $L_1, L_2$ over $S$ with $K_S = L_1L_2$,
\[
Y_{\phi,\psi} = \{(x,y,z) \in L_1^3 \oplus L_2^3 \oplus L_1L_2 \, | \, xy = z^3 + \phi\cdot z + \psi\}.
\] 
In this case the conjectural embedding of a moduli space of pairs $(\phi,\psi)$ into $\Stab(\scrF(Y_{\phi,\psi}))$ cannot be onto an open set for dimension reasons (cf. Section \ref{Subsec:Definitions}). 
The DT-counting invariants for semistable objects in $K$-theory class $d\gamma$ can have exponential growth in $d$ \cite{Galakhov_et_al} and have interesting algebraic generating functions \cite{Mainiero}. The 3-fold $Y_{\phi,\psi}$  now contains a  special Lagrangian submanifold $L\cong (S^1\times S^2)\#(S^1\times S^2)$, obtained from surgery of two 3-spheres lying over tripods which meet at three end-points, and the wild representation theory of $\pi_1(L)$ may be responsible for the exponential growth of stable objects (flat bundles over $L$) on the symplectic side. It would be interesting to know if \eqref{Hyp:finite_measure}  is related to polynomial growth of DT-invariants.


\subsection{Cubic four-folds}
The derived category of a cubic four-fold $Y\subset \bP^5$ admits a semi-orthogonal decomposition, the interesting piece of which is a CY2-category $\scrA_Y$ introduced by Kuznetsov \cite{Kuznetsov}.  These categories are of symplectic nature, cf. \cite[Proposition 2.17]{Huybrechts:K3category} and \cite{SS}.  Let $E$ be the (Fermat) elliptic curve with a non-trivial $\bZ/3$-action, generated by $\xi$, and $X$ be the K3 surface which is the crepant resolution of $(E\times E)/\langle(\xi,\xi^{-1})\rangle$.  (This is sometimes called the ``most algebraic" K3 surface; it has Picard rank 20 and, amongst such K3's, has smallest possible discriminant.)  Then for certain toric K\"ahler forms $\omega$ on $X$ which are ``irrational" (again meaning the areas of the resolution curves and a hyperplane section are linearly independent over $\bQ$), there is an equivalence \cite{Sheridan2016, SS} (strictly, this requires incorporating certain immersed Lagrangian tori into $\scrF(X,\omega)$)
\[
D^{\pi}\scrF(X,\omega) \simeq \scrA_{Y_{d(\omega)}} \, \subset \, D^b(Y_{d(\omega)})
\]
where the valuations of the coefficients in the equation defining the cubic $Y$ over $\Lambda$ are determined by the choice of K\"ahler form.   For sufficiently general $Y$, the space of stability conditions has been computed by Bayer \emph{et al} \cite{Bayer_et_al}, following a direct computation of autoequivalences due to Huybrechts \cite{Huybrechts:K3category}, and one finds (cf. \cite{SS, SS:in_progress}):

\begin{Theorem}[Sheridan, Smith] \label{Thm:most_algebraic}
For the most algebraic K3 surface $X$ with an irrational toric K\"ahler form, the map
$\rho: \pi_0\Symp(X,\omega) \longrightarrow \Auteq(\scrF(X,\omega))/[2]$ 
has image $\bZ/3$. If $Z(X,\omega) = \ker(\rho)$ then $\pi_0\Symp(X,\omega) = Z(X,\omega) \rtimes \bZ/3$.
\end{Theorem}
The image of $\rho$ is  generated by the obvious residual diagonal action on $E\times E$, which lifts to $X$. Theorem \ref{Thm:most_algebraic} reduces the task of understanding the symplectic mapping class group of $(X,\omega)$ to that of understanding what Floer theory doesn't see, which is about as far as we could hope to come (one can draw similar conclusions in the setting of Theorem \ref{Thm:SS}, see \cite{SS:k3}).  More concretely, for such irrational K\"ahler forms one learns that $HF(\phi)$ can take only one of three possibilties for any symplectomorphism $\phi$, and any $\phi$ has trivial Floer-theoretic entropy, etc.  From the K\"unneth theorem, this has consequences for fixed points of symplectomorphisms of $X\times T^2$, say, which -- in the same vein as the Arnol'd conjecture -- go beyond information extractable from smooth topology (stabilising by taking product with $T^2$ serves to kill information from the Lefschetz theorem).

It is interesting to imagine turning around the direction of mirror symmetry in this case.  Naively, one would predict some relation between the categories
\begin{equation} \label{eqn:cubic_speculation}
D^{\pi}\scrF(Y;0) \qquad \textrm{and} \qquad D^b(X) = D(E\times E)^{\bZ/3}
\end{equation}
where $\scrF(Y;0)$ is the nilpotent summand of the Fukaya category (corresponding in \eqref{eqn:splits} to the zero-eigenvalue).  It seems far from obvious that these should  be equivalent, and the actual relation, if any, might be more subtle. Nonetheless, although $\scrF(Y)$ is not $\bZ$-graded as $Y$ is Fano, the summand $\scrF(Y;0)$ \emph{is} expected to admit a $\bZ$-grading, whence one could talk about stability conditions.  Let $\scrM_{3,d}$ denote the moduli space of cubic $d$-folds.  These are well-studied spaces (and are famously CAT(0) for $d=2,3$ \cite{Allcock-Carlson-Toledo:cubic2, Allcock-Carlson-Toledo:cubic3}).  Starting from the lattice-theoretic co-incidence $H^2(X;\bZ) \supset \Pic(X)^{\perp} = -A_2 = \langle h^2\rangle^{\perp} \subset H^4(Y;\bZ)$, where $h^2\in H^4(Y;\bZ)$ is the class defined by a hyperplane, results of Laza \cite{Laza} yield an embedding (also observed by R.~Potter)
\[
\scrM_{3,4} \longrightarrow \Auteq_{CY}(X) \backslash \Stab_{\dagger}(D^b(X))/\widetilde{GL}^+(2,\bR)
\]
onto the complement of an explicit divisor $\Delta$ (associated to the locally finite hyperplane arrangement defined by classes of square  $-6$, or equivalently the complement in the period domain of the divisors associated to classes of square $-2$ and $-6$). Bridgeland's conjecture and some concrete relation in \eqref{eqn:cubic_speculation} might then give insight into the monodromy homomorphism
\[
\pi_1(\scrM_{3,4}) \longrightarrow \pi_0\Symp(Y) \longrightarrow \Auteq(D^{\pi}\scrF(Y;0)).
\]
In this way, one could hope to use stability conditions on $K3$ surfaces to attack classical problems related to the symplectic monodromy of hypersurfaces.

\subsection{Clusters}

Semantic sensitivities notwithstanding, we end with a digression. It may be useful to point out where much of the activity in the subject is concentrated.
If $(Q,W)$ satisfies suitable non-degeneracy assumptions, one can associate to $(Q,W)$ two spaces: $\Stab(\scrC(Q,W))$ and the cluster variety $\scrX(Q,W)$.  The first is glued together out of chambers $\hbar^d$ indexed by the vertices of the tilting tree, and the second is glued from birational maps of algebraic tori $(\bC^*)^d$ indexed by the same data.  It is believed that there is a (complicated, transcendental) complex Lagrangian submanifold $B\subset \Stab(\scrC)$ and an algebraic integrable system (with compact complex torus fibres) $\scrT \to B$ for which $\scrT$ and $\scrX$ are diffeomorphic, naturally equipped with different complex structures belonging to a single hyperk\"ahler family \cite{GMN, Neitzke:hk}.  The explicit diffeomorphism should be obtained from a Riemann-Hilbert problem, whose definition and solution involves the moduli stacks of stable objects and their Donaldson-Thomas theory \cite{Bridgeland:riemann_hilbert}.  

A genus $d$ Lagrangian surface $\Sigma_d \subset X^4$ in a symplectic four-manifold defines a chart $(\bC^*)^d$ in a tentative mirror to $X$, and one can use the complexity of cluster atlases to prove existence theorems for infinite families of Lagrangian surfaces \cite{STZ}. This is a striking connection back to symplectic topology, but one that it seems hard to formulate on the space of stability conditions directly.

\bibliographystyle{plain}
\bibliography{mybib}

\end{document}